\documentclass[12pt]{amsart}
\usepackage[top=25mm,bottom=15mm,left=20mm,right=20mm]{geometry}
\usepackage{mathrsfs}
\usepackage{amsfonts}
\usepackage{amssymb}
\usepackage{amsxtra}
\usepackage{color}
\usepackage{dsfont}
\usepackage{verbatim}

\usepackage[english,polish]{babel}

\newcommand{\sprod}[2]{\langle {#1}, {#2}\rangle}
\newcommand{\norm}[1]{{\lVert #1 \rVert}}

\newcommand{\R}{\mathbb{R}}
\newcommand{\PP}{\mathbb{P}}
\newcommand{\EE}{\mathbb{E}}

\newcommand{\Per}{\mathrm{Per}}

\newcommand{\ud}{\mathrm{d}}

\newcommand{\pl}[1]{\foreignlanguage{polish}{#1}}

\def\beq{\begin{equation}}
\def\eeq{\end{equation}}
\newcommand{\eps}{\varepsilon}
\def\pf{\noindent{\bf Proof.} }
\def\qed{{\hfill $\Box$ \bigskip}}
\def\nn{\nonumber}

\newtheorem{theorem}{Theorem}
\newtheorem{proposition}[theorem]{Proposition}
\newtheorem{lemma}[theorem]{Lemma}
\newtheorem{corollary}[theorem]{Corollary}
\newtheorem{example}{Example}
\newtheorem{remark}[theorem]{Remark}

\title{Spectral Heat Content for L\'{e}vy Processes}
\author{Tomasz Grzywny, Hyunchul Park, and Renming Song}
\thanks{T.~Grzywny was supported in part by  National Science Centre (Poland): grant 2016/23/B/ST1/01665.
R.~Song is partially supported by a grant from the Simons Foundation (\#429343, Renming Song).}

\address{
	Tomasz Grzywny\\
	 \pl{Wydzia{\lll}} Matematyki, 
	Politechnika \pl{Wroc{\lll}awska},
 	Wyb. \pl{Wyspia\'{n}skiego} 27,
 	50-370 \pl{Wroc\l{}aw},
 	Poland}
\email{tomasz.grzywny@pwr.edu.pl}

\address{
Hyunchul Park\\
Department of Mathematics, State University of New York at New Paltz, NY 12561,
USA
}
\email{parkh@newpaltz.edu}

\address{
Renming Song\\
Department of Mathematics, University of Illinois, Urbana, IL 61801,
USA
}
\email{rsong@illinois.edu}


\begin{document}
\selectlanguage{english}
\begin{abstract}
In this paper we study the spectral heat content for various L\'evy processes. 
We establish the asymptotic behavior of the spectral heat content for L\'evy processes of bounded variation in $\R^{d}$, $d\geq 1$. 
We also study the spectral heat content for arbitrary open sets of finite Lebesgue measure in $\R$ with respect to L\'evy processes of unbounded variation under certain conditions on their characteristic exponents. 
Finally we establish that the asymptotic behavior of the spectral heat content is stable under integrable perturbations to the L\'evy measure.
\end{abstract}

\maketitle

\section{Introduction}
Let $\mathbf{X}=(X_t)_{t\geq 0}$ be a  L\'{e}vy process in $\R^d$. For any open
set $\Omega\subset\R^d$, the (regular) heat content  of $\Omega$ with respect to $\mathbf{X}$
is defined to be
$$
H_{\Omega}(t)=\int_\Omega \PP_x(X_t\in\Omega)\ud x,
$$
while the spectral heat content of $\Omega$ with respect to 
$\mathbf{X}$ is defined to be
$$
Q_{\Omega}(t):=\int_{\Omega}\PP_x(\tau_{\Omega}>t)\ud x,
$$
where $\tau_{\Omega}=\inf\{t>0:X_t\in\Omega^c\}$ is the first time the process
$\mathbf{X}$ exits $\Omega$.

The asymptotic behavior of 
the heat content and the spectral heat content
has been studied intensively in the case of Brownian motion, see \cite{Preunkert} and \cite{vanDenBerg_4}--\cite{vanDenBerg5}. 
Recently significant progress has also been made in studying the heat content
and the spectral heat content with respect to L\'evy processes, see
\cite{Valverde2, Valverde1, Valverde3, cg, MRT}. The asymptotic behavior of the
heat content and the spectral heat content with respect to symmetric stable 
processes was studied in \cite{Valverde2, Valverde1, Valverde3}. In particular,
the exact asymptotic behavior of the spectral heat content of bounded open intervals with respect to symmetric stable processes in $\R$ was established in
\cite{Valverde3}. The asymptotic behavior of the heat content with respect to 
general L\'evy processes was studied in \cite{cg}, see also \cite{wctg} 
for a generalization. 
In \cite{MRT}, an asymptotic expansion of the heat content with respect to some isotropic compound Poisson
processes with compactly supported jumping kernels was established.

The purpose of this paper is to investigate the asymptotic behavior of the spectral heat content of general L\'evy processes and generalize the results of 
\cite{Valverde3} in several directions.

The organization of this paper is as follows.
In Section \ref{section:preliminaries} we recall some notions and present some preliminaries.
In Section \ref{section:bounded variation} we first study the heat content with respect to L\'evy processes of bounded variation. 
In Theorem \ref{thm:HK on rough sets}, we extend \cite[Theorem 3]{cg} to
any open set of
finite Lebesgue measure and relax the finite perimeter condition. 
Then we use this 
to establish the asymptotic behavior of the spectral heat content for the same processes in Theorems \ref{thm:main1} and \ref{thm:main7}. 
In Section \ref{section:unbounded variation} we investigate the asymptotic behavior of the spectral heat content with respect to L\'evy processes of unbounded variation in $\R$.
In this section we deal with two cases separately. 
In the first case, we assume that the characteristic exponent $\psi(\xi)$ of $\mathbf{X}$ is regularly varying of index $\alpha$ at infinity for some $\alpha\in (1,2]$.
In the second case, we assume that $\mathbf{X}$ is a symmetric 1-stable process, that is, a Cauchy process. 
The main results in Section \ref{section:unbounded variation} are Theorems \ref{thm:general open sets} and \ref{thm:mainCauchy}, where we establish the exact asymptotic behavior of the spectral heat content with respect to such processes. 
We note here that the asymptotic behavior of $Q_{\Omega}(t)$ depends on the geometry of $\Omega$. When $\alpha\in (1,2]$, both the number of adjacent components and the number of non-adjacent components matter, while in the case $\alpha=1$, only the number of non-adjacent components matters since the process can not hit a single point upon exiting the open set.  
Two components of $\Omega$ are said to be adjacent if the distance between them is zero.
In Section \ref{section:perturbation} we study the stability  of the spectral heat content. We prove in Theorem \ref{thm:main2} that the asymptotic behavior of the spectral heat content is stable under integrable perturbations to the L\'evy measures.
In Section \ref{section:examples} we give some examples where one can apply the results of this paper to get the asymptotic behavior of the spectral heat content.

\section{Preliminaries}\label{section:preliminaries}
Let $\mathbf{X}=(X_t)_{t\geq 0}$ be a  L\'{e}vy process in $\R^d$.
We denote by $(P_t)$ the semigroup of $\mathbf{X}$ and by $\widehat{P}_t$ the adjoint operator of $P_t$.
The characteristic exponent $\psi (\xi)$, $\xi\in \R ^d$, 
of $\mathbf{X}$ is given by 
\begin{align*}
\psi(\xi) = \sprod{\xi}{A\xi} - i\sprod{\xi}{\gamma } - 
\int_{\R ^d}\left(e^{i\sprod{\xi}{y}} - 1 - i\sprod{\xi}{y} 
1_{\{\norm{y} \leq 1\}}\right) \nu(\ud y) ,
\end{align*}
where $A$ is a symmetric non-negative definite $d\times d$ matrix, $\gamma\in \R ^d$ and $\nu$ is a L\'{e}vy measure, that is
\begin{align*}
\nu(\{0\})=0\quad \mathrm{and}\quad \int_{\R ^d}\min \{1, \norm{y}^2\}\, \nu(\ud y) <\infty .
\end{align*}

The L\'{e}vy process $\mathbf{X}$ is of bounded variation (see \cite[Theorem 21.9]{Sato}) if and only if
\begin{align*}
A=0\quad \mathrm{and}\quad \int_{\norm{y}\leq 1}\norm{y}\nu (\ud y)<\infty .
\end{align*}
In this case the characteristic exponent has the following simple form
\begin{align*}
\psi (\xi) = i\sprod{\xi}{\gamma _0} + \int_{\R^d}\left( 1-e^{i\sprod{\xi}{y}}\right)\nu (\ud y),
\end{align*}
where 
 $
\gamma _0 = \int_{\norm{y}\leq 1}y\, \nu (\ud y) - \gamma .
$ 
For a L\'{e}vy process of bounded variation, the quantity $\gamma_0$ defined above is called the drift of the process.

We introduce the following function related to $\mathbf{X}$, see \cite{Pruitt}. For any $r>0$,
\begin{align*}
h(r)= \norm{A}r^{-2} + 
 r^{-1} \Big\lvert \gamma +\int_{\R ^d} y \left(
1_{\{\norm{y} < r\}}-1_{\{\norm{y} < 1\}}\right)
 \nu(\ud y)\Big\rvert  
 +\int_{\R ^d} \min \{1, \norm{y}^2r^{-2} \}\, \nu(\ud y).
\end{align*}
Recall that there exists $C_1=C_1(d)>0$ (see \cite[page 941]{Pruitt}) such that
\begin{equation}\label{eq:Pruitt}
\PP(\sup_{s\leq t}|X_s|>r)\leq C_1th(r).
\end{equation}

Following \cite[Section 3.3]{Ambrosio_2000}, for any  
Borel set $\Omega \subset \R ^d$, we define its perimeter $\Per(\Omega)$ as
\begin{align*}
\Per(\Omega) = \sup \left\{ \int_{\R^d}
1_{\Omega}(x)\mathrm{div}\, \phi (x)\, \ud x:
\, \phi \in C_c^1(\R^d,\R^d),\, \norm{\phi}_{\infty}\leq 1 \right\}.
\end{align*}
We say that $\Omega$ is of finite perimeter if $\Per(\Omega)<\infty$. 
It was shown \cite{Miranda1, Miranda2, Preunkert} that if $\Omega\subset \R^d$ is an open set of  finite Lebesgue measure and of finite perimeter, then
\begin{align*}
\Per(\Omega) = \pi^{1/2}\lim_{t\to 0}t^{-1/2}\int_{\Omega}\int_{\Omega^c}p_t^{(2)}(x,y)\, \ud y\, \ud x,
\end{align*}
where
\begin{align*}
p_t^{(2)}(x,y) = (4\pi t)^{-d/2}e^{-\norm{x-y}^2/4t}
\end{align*}
is the transition density of the
Brownian motion $\mathbf{B}=(B_t)_{t\ge 0}$ in $\R^d$. 
For a L\'{e}vy process $\mathbf{X}$ with L\'{e}vy measure $\nu$, 
we define the perimeter $\Per_{\mathbf{X}}(\Omega)$ with respect to 
$\mathbf{X}$ as
\begin{align*}
\Per_{\mathbf{X}}(\Omega)= \int_{\Omega}\int_{\Omega ^c-x}\nu (\ud y)\, \ud x .
\end{align*}
In particular, to the isotropic (rotationally invariant) $\alpha$-stable 
process $\mathbf{S}^{(\alpha)}=(S^{(\alpha)}_t)_{t\geq 0}$, $0<\alpha <1$, one associates the $\alpha$-perimeter which is defined as
\begin{align*}
\Per_{\mathbf{S}^{(\alpha)}}(\Omega) :=
\int_{\Omega}\int_{\Omega ^c} 
\frac{c(d,\alpha)\ud y\, \ud x}{\norm{x-y}^{d+\alpha}},
\end{align*}
where
$$
c(d, \alpha):= \frac{ \alpha  \,
\Gamma(\frac{d+\alpha}2)} {2^{1-\alpha} \,
\pi^{d/2}\Gamma(1-\frac{\alpha}2)}.
$$
It is known, cf. \cite[Lemma 1]{cg} (see also \cite{Fusco} for the perimeter for the isotropic stable processes), that if $\mathbf{X}$ is of bounded variation
and $\Omega\subset \R^d$ is an open set of finite  Lebesgue measure and of
finite perimeter, then $\Per_{\mathbf{X}}(\Omega)$ is also finite.

Let $G\subset \R^{d}$ be an open set  
and $f:G\rightarrow \R$ be integrable.
The total variation of $f$ in $G$ is 
$$
V(f,G)=\sup\{\int_{G}f(x)\text{div}\varphi(x)\ud x : \varphi\in C_{c}^{1}(G,\R^{d}), \|\varphi\|_{\infty}\leq 1\}.
$$
The directional derivative of $f$ in $G$ in the direction of $u\in\mathbb{S}^{d-1}$ is 
$$
V_{u}(f,G)=\sup\{\int_{G}f(x)\langle\nabla\varphi(x),u\rangle\ud x : \varphi\in C_{c}^{1}(G,\R^{d}), \|\varphi\|_{\infty}\leq 1\}.
$$
We will use $V_{u}(\Omega)$ to denote $V_{u}(1_{\Omega},\R^{d})$.

Now we recall the covariogram function $g_{\Omega}(y)$. Let 
$$
g_{\Omega}(y)=|\Omega\cap (\Omega+y)|=\int_{\R^{d}}1_{\Omega}(x)1_{\Omega+y}(x)\ud x=\int_{\R^{d}}1_{\Omega}(x)1_{\Omega}(x-y)\ud x.
$$
It is easy to see that
$$
g_{\Omega}(y)\leq g_{\Omega}(0)=|\Omega| \quad \text{ and } \quad g_{\Omega}(-y)=g_{\Omega}(y).
$$
Moreover $g_\Omega\in C_0(\R^d)$ (see \cite[Proposition 2]{Galerne}).

Let
$$
f_{\Omega}(y):=g_{\Omega}(0)-g_{\Omega}(y).
$$
By the Fubini-Tonelli theorem we have the following relationship between 
$\text{Per}_{\mathbf{X}}(\Omega)$ and $f_{\Omega}(\cdot)$:
\begin{equation}\label{eqn:perimeter}
{\rm Per}_{\mathbf{X}}(\Omega)=\int_{\R^{d}}f_{\Omega}(y)\nu(\ud y).
\end{equation}

Here is a simple lemma about the behavior of $f_{\Omega}(y)$ as $|y|\rightarrow 0$.
\begin{lemma}\label{lemma:limit=0}
If $\Omega\subset \R^{d}$ is an open set of finite Lebesgue measure $|\Omega|$, then
$$
\lim_{|y|\rightarrow 0}f_{\Omega}(y)=0.
$$
\end{lemma}
\pf
Note that $1_{\Omega}(x)1_{\Omega^{c}}(x-y)\leq 1_{\Omega}(x)$ for all $y\in \R^{d}$ and $\int1_{\Omega}(x)\ud x=|\Omega|< \infty$.
For each $x\in \Omega$ we have
$$
\lim_{|y|\rightarrow 0} 1_{\Omega}(x)1_{\Omega^{c}}(x-y)=0.
$$
Hence the assertion follows from the dominated convergence theorem.
\qed

The next lemma follows from \cite[Proposition 5]{Galerne}.
\begin{lemma}\label{lemma:triangular inequality}
If $\Omega\subset \R^{d}$ is a
Borel  set of finite Lebesgue measure, then
$$
\left|g_{\Omega}(x)-g_{\Omega}(y)\right|\leq g_{\Omega}(0)-g_{\Omega}(x-y), \qquad x,y\in \R^{d}.
$$
\end{lemma}

We end this section by recalling the concept of regularly varying functions. A function
$f$ is said to be regularly varying of index $\alpha$ at infinity if for any $\lambda>0$,
$$
\lim_{r\to\infty}\frac{f(\lambda r)}{f(r)}=\lambda^\alpha.
$$
The family of  regularly varying functions of index $\alpha$ at infinity is denoted by
$\mathcal{R}_{\alpha}$.

\section{Processes of bounded variation in $\R^d$}\label{section:bounded variation}
In this section we assume that $\mathbf{X}$ is a L\'{e}vy process of bounded variation in $\R^d$. 
\subsection{Heat content}
In this subsection, we first extend \cite[Theorem 3]{cg} when the drift $\gamma_0=0$.
Suppose that $\mathbf{X}$ is a purely discontinuous  L\'evy 
process of bounded variation in $\R^d$, that is,  $A=0$, $\gamma_0=0$ and $\int_{\norm{x}\leq 1}\norm{x}\nu (\ud x)<\infty$.

The infinitesimal generator  $\mathcal{L}$ on $C_0(\R^d)$ of $\mathbf{X}$ is
a linear operator defined by
\begin{align}\label{gener}
\mathcal{L}f(x) = \lim_{t\to 0}\frac{ P_tf(x) - f(x) }{t},
\end{align}
with domain $\mathrm{Dom}(\mathcal{L})$ consisting of all functions $f$ such that the right hand side of \eqref{gener} exists.
By \cite[Theorem 31.5]{Sato}, we have $C_0^2(\R^d)\subset \mathrm{Dom}(\mathcal{L})$. 
For a detailed discussion on infinitesimal generators of L\'{e}vy processes we refer the reader to \cite[Section 31]{Sato}.
Since $\mathbf{X}$ is of bounded variation and $\gamma_0=0$, again by \cite[Theorem 31.5]{Sato} we have,  for any $f\in C_0^2(\R^d)$,
\begin{align}\label{eq:gener_finvar}
\mathcal{L} f(x) &= \int_{\R^d} \left(f(x+y)-f(x)\right) \nu(\ud y).
\end{align}

The next lemma  corresponds to \cite[Lemma 2]{cg}.
\begin{lemma}\label{lemma:domain}
Suppose that $\Omega\subset \R^{d}$ is an open set of finite Lebesgue measure. If
$$
\int_{\R^{d}}f_{\Omega}(y)\nu(\ud y)<\infty, 
$$
then $g_{\Omega}(x)\in 
{\rm Dom}(\mathcal{L})$ and 
$$
\mathcal{L}g_{\Omega}(x)=\int_{\R^{d}}\left(g_{\Omega}(x+y)-g_{\Omega}(x)\right)\nu(\ud y).
$$
\end{lemma}
\pf
Fix a cut-off function $\phi\in C_{c}^{\infty}(\R^{d})$  such that $\phi\geq0$, $\mathrm{supp}\phi\subset B(0,1)$ and  $\|\phi\|_{1}=1$.  Let $\phi_{\eps}(x):=\frac{1}{\eps^{d}}\phi(\frac{x}{\eps})$.
Let $g_{\Omega}^{\eps}(x):=g_{\Omega}*\phi_{\eps}(x)$.
Since $g_\Omega$ is integrable, we get that $g_{\Omega}^{\eps}$ is smooth and vanishes at infinity. Hence $f_{\Omega}^{\eps}\in C_{0}^{\infty}(\R^{d})$.
Since 
$C_{0}^{2}(\R^{d})\subset \text{Dom}(\mathcal{L})$, we get 
$g_{\Omega}^{\eps}\in \text{Dom}(\mathcal{L})$.
Note that it follows from Lemma \ref{lemma:triangular inequality} that
\begin{align*}
|g_{\Omega}^{\eps}(x)-g_{\Omega}(x)|
=\left|\int_{\R^{d}}\phi_{\eps}(y)\left(g_{\Omega}(x-y)-g_{\Omega}(x)\right)\ud y\right|
&\leq \sup_{|y|\leq\eps}|g_\Omega(0)-g_{\Omega}(y)\|\phi_{\eps}\|_{1}
=\sup_{|y|\leq\eps}|f_{\Omega}(y)|.
\end{align*}
 Lemma \ref{lemma:limit=0} implies  
$$
\lim_{\eps\downarrow 0}||g_{\Omega}^{\eps}-g_{\Omega}||_{\infty}=0.
$$
By \eqref{eq:gener_finvar},
$$\mathcal{L}g_{\Omega}^{\eps}(x)=\int_{\R^d}(g_{\Omega}^\eps(x+y)-g_{\Omega}^\eps(x))\nu(\ud y).$$
Note that from Lemma \ref{lemma:triangular inequality} we have
$$
|g^\eps_\Omega(x+y)-g^\eps_{\Omega}(x)|\leq f_\Omega(y),\quad x,y\in\R^d.
$$ 
Hence,
by the dominated convergence theorem  we infer that
$$
\lim_{\eps\downarrow 0}\|\mathcal{L}g_{\Omega}^{\eps}-\int_{\R^{d}}\left(g_{\Omega}(\cdot+y)-g_{\Omega}(\cdot)\right)\nu(\ud y)\|_\infty=0.
$$
Since $\mathcal{L}$ is a closed operator, the assertion of the lemma is now established. 
\qed

The following result is similar in spirit to \cite[Theorem 3]{cg}.
The difference is that in the result below we do not assume ${\rm Per}(\Omega)<\infty$ but we assume that $\gamma_0=0$.

\begin{theorem}\label{thm:HK on rough sets}
Let $\mathbf{X}$ be a L\'{e}vy process of bounded variation with $\gamma_0=0$.
If $\Omega\subset \R^{d}$ is an open set of finite Lebesgue measure,
then we have
$$
\lim_{t\rightarrow 0}\frac{|\Omega|-H_{\Omega}(t)}{t}={\rm Per}_{\mathbf{X}}(\Omega).
$$
\end{theorem}
\pf
First we assume that 
$${\rm Per}_{\mathbf{X}}(\Omega)=\int_{\Omega}\nu(\Omega^c- y)\ud y<\infty.$$
In this case the proof is similar to that of \cite[Theorem 3]{cg}.
It follows from Lemma \ref{lemma:domain} and \eqref{eqn:perimeter} that
\begin{align*}
&\lim_{t\rightarrow 0}\frac{|\Omega|-H_{\Omega}(t)}{t}=\lim_{t\rightarrow 0}\int_{\Omega}\frac{1-\PP(x+X_{t}\in \Omega)}{t}\ud x\\
&=\lim_{t\rightarrow 0}\frac{g_{\Omega}(0)-P_{t}g_{\Omega}(0)}{t}=-\mathcal{L}g_{\Omega}(0)=\int_{\R^{d}}f_{\Omega}(y)\nu(\ud y)={\rm Per}_{\mathbf{X}}(\Omega).
\end{align*}

Now we deal with the case when
$$
\int_{\Omega}\nu(\Omega^c- y)\ud y=\infty.
$$
Note that
$$
\frac{|\Omega|-H_{\Omega}(t)}{t}=\int_{\R^{d}}f_{\Omega}(y)\frac{p_t(\ud y)}{t}.
$$
Let $\varepsilon>0$, and $\phi_\varepsilon \in C_b(\R^d)$ be  such  that $1_{B(0,\eps)^c}\leq \phi_\eps\leq 1_{ B(0,\eps/2)^c}$.
Then by \cite[Corollary 8.9]{Sato} we get
$$
\liminf_{t\rightarrow 0}\frac{|\Omega|-H_{\Omega}(t)}{t}\geq \liminf_{t\rightarrow 0}\int_{\R^{d}}\phi_\eps(y)f_{\Omega}(y)\frac{p_t(\ud y)}{t}=\int_{\R^{d}}\phi_\eps(y)f_{\Omega}(y)\nu(\ud y)\geq \int_{B(0,\eps)^c}f_{\Omega}(y)\nu(\ud y).
$$
Since $\eps$ is arbitrary, we have
$$
\liminf_{t\rightarrow 0}\frac{|\Omega|-H_{\Omega}(t)}{t}=\infty.$$
\qed

\subsection{Spectral heat content}
In this subsection we study the asymptotic behavior of the spectral heat content for L\'evy processes of bounded variation.
The main result is the following theorem.

\begin{theorem}\label{thm:main1}
Let $\mathbf{X}$ be a L\'{e}vy process of bounded variation in $\R^d$.
If $\Omega \subset \R^d$ is an open set of finite measure $|\Omega| $ and of finite perimeter $\Per (\Omega)$,  then
\begin{align*}
\limsup_{t\to 0}\frac{|\Omega|-Q_{\Omega}(t)}{t} = &\Per_{\mathbf{X}}(\Omega) +
\norm{\gamma_0}V_{\frac{\gamma_0}{\norm{\gamma_0}}}(\Omega)/2\\
&+\limsup_{t\to0}\frac{1}{t}\int_{\Omega}\EE_x\left(\tau_{\Omega}<t, 
X_{\tau_{\Omega}}\in \partial\Omega;\PP_{X_{\tau_{\Omega}}}
\left(X_{t-\tau_{\Omega}}\in \Omega\right)\right)\ud x,\\
\liminf_{t\to 0}\frac{|\Omega|-Q_{\Omega}(t)}{t} = &\Per_{\mathbf{X}}(\Omega) +
\norm{\gamma_0}V_{\frac{\gamma_0}{\norm{\gamma_0}}}(\Omega)/2\\
&+\liminf_{t\to0}\frac{1}{t}\int_{\Omega}\EE_x\left(\tau_{\Omega}<t, 
X_{\tau_{\Omega}}\in \partial\Omega;\PP_{X_{\tau_{\Omega}}}
\left(X_{t-\tau_{\Omega}}\in \Omega\right)\right)\ud x,
\end{align*}
where $V_u(\Omega)$ is the directional derivative of 
$1_{\Omega}$ in the direction $u$ on the unit sphere in $\R^d$.
\end{theorem}
\begin{proof}
Observe that
\begin{align*}
|\Omega|-Q_{\Omega}(t)&=\int_{\Omega}\left(1-\PP_x(X_t\in \Omega)\right)\ud x+\int_{\Omega}\EE_x\left(\tau_{\Omega}<t;\PP_{X_{\tau_{\Omega}}}\left(X_{t-\tau_{\Omega}}\in \Omega\right)\right)\ud x\\&=
(|\Omega|-H_{\Omega}(t))+\int_{\Omega}\EE_x\left(\tau_{\Omega}<t;\PP_{X_{\tau_{\Omega}}}\left(X_{t-\tau_{\Omega}}\in \Omega\right)\right)\ud x\\&=(|\Omega|-H_{\Omega}(t))+\mathrm{I}(t)+\mathrm{II}(t),
\end{align*}
where
\begin{align*}
\mathrm{I}(t)&=\int_{\Omega}\EE_x\left(\tau_{\Omega}<t, X_{\tau_{\Omega}}\in \overline{\Omega}^c;\PP_{X_{\tau_{\Omega}}}\left(X_{t-\tau_{\Omega}}\in \Omega\right)\right)\ud x,\\
\mathrm{II}(t)&=\int_{\Omega}\EE_x\left(\tau_{\Omega}<t, X_{\tau_{\Omega}}\in \partial\Omega;\PP_{X_{\tau_{\Omega}}}\left(X_{t-\tau_{\Omega}}\in \Omega\right)\right)\ud x.
\end{align*}
By \cite[Theorem 3]{cg}, it suffices to show that
$$
\lim_{t\to0} \frac{\mathrm{I}(t)}{t}=0.
$$
By the Ikeda-Watanabe formula \cite{MR0142153},  the joint distribution of $(\tau_\Omega,X_{\tau_{\Omega}})$ restricted to $X_{\tau_\Omega-}\neq X_{\tau_\Omega}$ is equal to
$$
\PP_x((\tau_\Omega,X_{\tau_{\Omega}})\in(\ud s, \ud z))=\int_{\Omega}
p^\Omega_{s}(x,\ud u)\nu(\ud z-u)\ud s,
$$
where $p^\Omega_{s}(x,\ud u)$ is the transition kernel of the process $\mathbf{X}$ killed upon exiting $\Omega$.
Hence
\begin{align*}
\mathrm{I}(t)&=\int_{\Omega}\ud x\int^t_0\ud s \int_{\overline{\Omega}^c}\PP_z(X_{t-s}\in \Omega)\int_{\Omega}p^\Omega_s(x,\ud u)\nu(\ud z-u)\\
&\leq \int^t_0\ud s \int_\Omega  P_s (g_{t-s})(x)\ud x ,
\end{align*}
where 
$$g_{s}(u)=\textbf{1}_{\Omega}(u)\int_{\overline{\Omega}^c-u}\PP_{z+u}(X_{s}\in \Omega)\nu(\ud z).
$$
Notice that
$$
\int_\Omega g_s(x)\ud x\leq \int_\Omega\int_{\Omega^c-x}\nu(\ud z)<\infty.
$$
Thus
\begin{align*}
\mathrm{I}(t)&\leq \int^t_0\ud s \int_\Omega \widehat{P}_s(\textbf{1}_\Omega)(x)g_{t-s}(x)\ud x\leq \int^t_0\ud s\int_\Omega g_s(x)\ud x.
\end{align*}
By the dominated convergence theorem and the right continuity of $\mathbf{X}$,
$$
\limsup_{t\to0}\frac{\mathrm{I}(t)}{t}\leq \limsup_{t\to0}\int_{\Omega}g_t(x)\ud x=\int_{\Omega}\limsup_{t\to0}g_t(x)\ud x=0.
$$
\end{proof}

As a consequence of Theorem \ref{thm:HK on rough sets}, one can prove  the following result.
Note that, unlike Theorem \ref{thm:main1}, we do not assume that 
$\Per (\Omega)<\infty$ in the result below.

\begin{theorem}\label{thm:main7}
Let $\mathbf{X}$ be a L\'{e}vy process of bounded variation with $\gamma_0=0$.
If $\Omega\subset \R^{d}$ is an open set of  finite Lebesgue measure,
then we have 
$$
\limsup_{t\rightarrow 0}\frac{|\Omega|-Q_{\Omega}(t)}{t}={\rm Per}_{\mathbf{X}}(\Omega)+\limsup_{t\rightarrow 0}\frac{1}{t}\int_{\Omega}\EE_{x}[\tau_{\Omega}<t, X_{\tau_{\Omega}}\in \partial \Omega, \PP_{X_{\tau_{\Omega}}}(X_{t-\tau_{\Omega}}\in \Omega)]\ud x
$$
and
$$
\liminf_{t\rightarrow 0}\frac{|\Omega|-Q_{\Omega}(t)}{t}={\rm Per}_{\mathbf{X}}(\Omega)+\liminf_{t\rightarrow 0}\frac{1}{t}\int_{\Omega}\EE_{x}[\tau_{\Omega}<t, X_{\tau_{\Omega}}\in \partial \Omega, \PP_{X_{\tau_{\Omega}}}(X_{t-\tau_{\Omega}}\in \Omega)]\ud x.
$$
In particular if $\PP^x(X_{\tau_\Omega}\in\partial\Omega)=0$ for almost every $x\in \Omega$ we have
$$\lim_{t\rightarrow 0}\frac{|\Omega|-Q_{\Omega}(t)}{t}={\rm Per}_{\mathbf{X}}(\Omega).$$
\end{theorem}
\pf
If $\int_{\Omega}\nu(\Omega^c- y)\ud y<\infty$, the proof is the same as the proof of Theorem \ref{thm:main1} using Theorem \ref{thm:HK on rough sets} instead of \cite[Theorem 3]{cg}. 
The case of $\int_{\Omega}\nu(\Omega^c- y)\ud y=\infty$  is a consequence of Theorem  \ref{thm:HK on rough sets} and the fact that 
$H_\Omega(t)\geq Q_{\Omega}(t)$.

\qed

Combining the above result with \cite[Theorem 1]{Sztonyk}, we immediately get the 
following

\begin{corollary}\label{cor:IsoSHC}
Suppose that $\mathbf{X}$ is an isotropic L\'evy process of bounded variation and has an infinite L\'evy measure.
If $\Omega$ is a Lipschitz domain of finite Lebesgue measure, then 
$$
\lim_{t\to 0}\frac{|\Omega|-Q_{\Omega}(t)}{t} = \Per_{\mathbf{X}}(\Omega).
$$
\end{corollary}

Combining Theorem \ref{thm:main7} with \cite{Bret}, we get the following
\begin{corollary}\label{cor:HK 1d}
Let $d=1$ and $\Omega$ be open. 
Assume that $\mathbf{X}$ has an infinite L\'{e}vy measure, is of bounded variation and 
$\gamma_0=0$, then
\begin{align*}
\lim_{t\to 0}\frac{|\Omega|-Q_{\Omega}(t)}{t} &= \Per_{\mathbf{X}}(\Omega).
\end{align*}
\end{corollary}
\pf
We use here that $\{0\}$ is polar (\cite[Theoreme 8]{Bret}) and therefore $\partial\Omega$ is polar.
\qed

If there exists a nonzero drift $\gamma_{0}$, then the asymptotic behavior of
the  heat content and the spectral heat content can be different. We illustrate 
this by the simple example below. 

\begin{example}
Let $\gamma\in\mathbb{R}\setminus\{0\}$. We consider $\Omega=(-1,0)\cup(0,1)$ and  a deterministic process $X_t=\gamma t$. Then $H(t)=(|\gamma|t)\wedge 2$ and $|\Omega|-Q_\Omega(t)=(2|\gamma|t)\wedge 2$. That is 
$$ 2|\gamma|=\lim_{t\to 0}\frac{|\Omega|-Q_\Omega(t)}{t}\neq \lim_{t\to 0}\frac{H(t)}{t}=|\gamma|.$$
\end{example}

\section{Processes of unbounded variation in $\R$}\label{section:unbounded variation}
In this section we study the asymptotic behavior of the spectral heat content for symmetric L\'evy processes on the real line. 
We consider two different cases separately. 
In the first case, we assume that the characteristic exponent $\psi(\xi)$ of $\mathbf{X}$ is regularly varying of index $\alpha$ at infinity for some $\alpha\in (1,2]$.
In the second case, we assume that $\mathbf{X}$ is a symmetric L\'evy process whose the characteristic exponent is $\psi(\xi)=|\xi|$, 
that is,  $\mathbf{X}$ is a Cauchy process. 

For any $\epsilon>0$, let $\Omega_\epsilon:=\{x\in \Omega: \mathrm{dist}(\{x\},\partial\Omega)<\epsilon\}$.
\begin{lemma}\label{lemma:inside}
For any $\eps>0$, we have
$$
\int_{\Omega\setminus \Omega_{\eps}}
\PP_{x}(\tau_{\Omega}\leq t)\ud x \leq C_1\,
|\Omega\setminus \Omega_{\eps}|th(\eps).
$$ 
\end{lemma}
\pf
It follows from \eqref{eq:Pruitt} that
$$
\int_{\Omega\setminus \Omega_{\eps}}\PP_{x}(\tau_{\Omega}\leq t)\ud x
\leq\int_{\Omega\setminus \Omega_{\eps}}\PP(\sup_{s\leq t}X_{s}\geq \eps)\ud x
\leq C_1\,|\Omega\setminus \Omega_{\eps}|th(\eps). 
$$
\qed

Every open set $\Omega$ in $\R$ can be written as the union of countably many
disjoint open intervals: $\Omega=\cup_{i}(a_{i},b_{i})$. 
Let $\Omega=\cup_{i}(a_{i},b_{i})$ and $\partial^{ad}\Omega:=\{b_{j}: 
\text{ there exists } i\neq j \text{ such that } b_{j}=a_{i}\}$ 
be the subset of $\partial \Omega$ 
which consists of  common boundary points of adjacent
components of $\Omega$.
Let
\beq\label{eqn:augmented set}
\tilde{\Omega}:=\Omega\cup \partial^{ad}\Omega
\eeq
be the augmented set of $\Omega$.
Note that the distance between any two
distinct components of $\tilde{\Omega}$ is always strictly positive.

Recall that
$$
f_{\Omega}(y)=|\Omega|-|\Omega\cap (\Omega+y)|=\int_{\R}1_{\Omega}(x)\ud x-\int_{\R}1_{\Omega}(x)1_{\Omega}(x-y)\ud x=\int_{\R}1_{\Omega}(x)1_{\Omega^{c}}(x-y)\ud x.
$$
For any L\'evy process $\mathbf{X}$ and $t\ge 0$, we define 
$\overline{X}_t:=\sup_{s\in [0, t]}X_s$ and $\underline{X}_t:=\inf_{s\in [0, t]}X_s$.

Let $\psi^{*}(u)=\sup_{\xi\in[0,u]}\psi(\xi)$, $u\geq 0$, and let 
$\psi^{-1}(u)=\inf\{s\geq 0 : \psi^{*}(s)\geq u\}$ be the generalized inverse of $\psi^{*}$.

\subsection{$\psi\in \mathcal{R}_{\alpha}$, $\alpha\in (1,2]$}
In this subsection we study 
the asymptotic behavior of the spectral heat content of general open sets of finite Lebesgue measure with respect to 
symmetric L\'evy processes in $\R$.
We assume that $\mathbf{X}$ is a symmetric L\'evy process with the characteristic
exponent $\psi$ and that there exists $\alpha\in (1,2]$ such  that $\psi\in \mathcal{R}_{\alpha}$.

For $y\in \R$, let $T_{y}=\inf\{t>0 : X_{t}=y\}$ be the first time  the process $\mathbf{X}$ hits $y$ and $T^{(\alpha)}_{y}$ the first time the symmetric
$\alpha$-stable process $\mathbf{S}^{(\alpha)}$ 
(with characteristic exponent $|\xi|^\alpha$) hits $y$.

Here is the main result of this section:

\begin{theorem}\label{thm:general open sets}
Suppose that $\mathbf{X}$ is a symmetric L\'evy process with the characteristic
exponent $\psi\in\mathcal{R}_\alpha$ for some $\alpha\in(1,2]$.
Let $\Omega$ be an open set in $\R$ with $|\Omega|<\infty$.
Let $A$ be the number of components of $\tilde{\Omega}$ and $B$ be number 
of points in $\partial^{ad}\Omega$.
Then we have
\beq\label{eqn:main}
\lim_{t\rightarrow 0}\psi^{-1}(1/t)\left(|\Omega|-Q_{\Omega}(t)\right)=
2A\EE[\overline{S^{(\alpha)}}_{1}]
+2BC_{1},
\eeq
where $C_{1}=\int_{0}^{\infty}\PP(T_{u}^{(\alpha)}\leq 1)\ud u\leq \EE[\overline{S^{(\alpha)}}_{1}]<\infty$.
\end{theorem}

In the case of 
isotropic $\alpha$-stable processes, we have
$\psi^{-1}(1/t)=t^{-1/\alpha}$.

\begin{remark}
Note that $A$ can be finite even if the number of components in $\Omega$ is infinite. 
Also if $\Omega$ has infinitely many components, either $A$ or $B$ must 
be infinite and therefore we have 
$\displaystyle\lim_{t\rightarrow 0}\psi^{-1}(1/t)\left(|\Omega|-Q_{\Omega}(t)\right)=\infty$.

Under the assumptions of this subsection, the process $\mathbf{X}$ has a transition
density, and thus $H_{\Omega}(t)=H_{\tilde{\Omega}}(t)$.
Hence it follows from \cite[Theorem 2]{cg} that when $\Omega\subset \R$ has infinitely many components but $\tilde{\Omega}$ has only finitely many components, we have
$$
\lim_{t\to 0}\psi^{-1}(1/t)(|\Omega|-H_\Omega(t))
=\frac{\Gamma(1-\frac{1}{\alpha})}{\pi}\rm{Per}(\Omega)=\frac{2\Gamma(1-\frac{1}{\alpha})}{\pi}A<\infty.
$$ 
\end{remark}

\begin{remark}
In the case of Brownian motion,
we have
$$
C_{1}=\int_{0}^{\infty}\PP(T_{u}^{\mathbf{B}}\leq 1)\ud u=\int_{0}^{\infty}\PP(\overline{B}_{1}\geq u)\ud u=\EE[\overline{B}_{1}].
$$
Hence in this case \eqref{eqn:main} becomes 
$$
\lim_{t\rightarrow 0}\frac{|\Omega|-Q_{\Omega}^{\mathbf{B}}(t)}{\sqrt{t}}
=2\times (\text{the number of components in } \Omega)\times \frac{2}{\sqrt{\pi}}.
$$
\end{remark}

We now give some preliminary results to prepare for the proof of Theorem \ref{thm:general open sets}.
It is easy to check that in this case $\mathbf{X}$ is of unbounded variation.
Since $\mathbf{X}$ is symmetric, it follows from \cite[Corollary 1]{Grzywny1} that
\beq\label{eqn:comparable1}
\frac12 \psi^{*}(r^{-1})\leq h(r)\leq 
24\psi^{*}(r^{-1}).
\eeq
Also it follows from \cite[Theorem 1.5.3]{bgt} that, for each $R>0$, there exists a constant $c=c(R)>0$ such that
\beq\label{eqn:comparable2}
\psi^{*}(r)\leq c\psi(r), \quad \text{for } r\geq  R.
\eeq

Since $\psi\in\mathcal{R}_\alpha$ for some $\alpha\in (1,2]$ and $\mathbf{X}$ is symmetric, we have
$$
\int_{\R}\frac{1}{1+\psi(\xi)}\ud\xi<\infty.
$$
It follows from \cite[Theoreme 8]{Bret} 
that $\{y\in \R : \PP(T_{y}<\infty)>0\}=\R$.
\begin{lemma}\label{lemma:hitting a point}
Suppose that $\mathbf{X}$ is a symmetric L\'evy process with the characteristic
exponent $\psi\in\mathcal{R}_\alpha$ for some $\alpha\in(1,2]$.
There exists $\eps_1>0$ such that for any  $\eps\leq \eps_1$,
$$
\lim_{t\rightarrow 0}\psi^{-1}(1/t) \int_{0}^{\eps}\PP(T_{y}\leq t)\ud y
=\int_{0}^{\infty}\PP(T_{u}^{(\alpha)}\leq 1)\ud u.
$$
\end{lemma}
\pf Define a process $\mathbf{Y}^{(t)}=(Y^{(t)}_{s})_{s\ge 0}$ by 
$Y^{(t)}_{s}=\psi^{-1}(1/t)X_{ts}$. 
The characteristic exponent of $\mathbf{Y}^{(t)}$ is
$\psi^{(t)}(\xi)=t\psi(\psi^{-1}(1/t)\xi)$.
Note that
$$
\lim_{t\to 0}\psi^{(t)}(\xi)=\lim_{t\to 0}\frac{\psi(\psi^{-1}(1/t)\xi)}{\psi(\psi^{-1}(1/t))}=|\xi|^\alpha.
$$
Observe that by the change of variables $u=\psi^{-1}(1/t)y$,
\begin{align*}
\psi^{-1}(1/t)\int_{0}^{\eps}\PP(T_{y}\leq t)\ud y
=\psi^{-1}(1/t)\int_{0}^{\eps}\PP(T_{\psi^{-1}(1/t)y}^{\mathbf{Y}^{(t)}}\leq 1)\ud y
=\int_{0}^{\eps\psi^{-1}(1/t)}\PP(T_{u}^{\mathbf{Y}^{(t)}}\leq 1)\ud u.
\end{align*}
Fix $0<\delta<\alpha-1$.
Since $\psi\in \mathcal{R}_{\alpha}$, it follows from \cite[Theorem 1.5.6]{bgt} that there exists $x_0>0$ such that
\beq\label{eqn:LDCT1}
\frac{\psi(\psi^{-1}(1/t)1/u)}{\psi(\psi^{-1}(1/t))}\leq \frac{2}{u^{\alpha-\delta}}
\eeq
for all $\psi^{-1}(1/t)\geq x_{0}$ and $1\leq u\leq \frac{\psi^{-1}(1/t)}{x_{0}}$.
Hence by \cite[Theorem II.19.(iii)]{Ber} and the dominated convergence theorem,
\beq\label{eqn:hitting time conv}
\lim_{t\rightarrow 0}\PP(T_{u}^{\mathbf{Y}^{(t)}}\leq 1)=\PP(T_{u}^{(\alpha)}\leq 1).
\eeq

Let $\eps_1:=\frac{1}{x_{0}}$.
It follows from \eqref{eqn:comparable2} 
that there exists $c=c(x_{0})>0$ such that
\beq\label{eqn:LDCT2}
\psi^{*}(r)\leq c(x_{0})\psi(r), \quad r\geq x_{0}.
\eeq
Take $M>0$.
It follows from 
the dominated convergence theorem and \eqref{eqn:hitting time conv} that
\beq\label{eqn:up to M}
\lim_{t\rightarrow 0}\int_{0}^{M}\PP\left(T_{u}^{\mathbf{Y}^{(t)}}\leq 1\right)\ud u=\int_{0}^{M}\PP(T_{u}^{(\alpha)}\leq 1)\ud u.
\eeq
Suppose $\eps\leq \eps_1$ and $M\leq u \leq \eps \psi^{-1}(1/t)$. Then from \eqref{eq:Pruitt}, \eqref{eqn:comparable1}, \eqref{eqn:LDCT1} and \eqref{eqn:LDCT2} we get that
\begin{align*}
&\PP\left((T_{u}^{\mathbf{Y}^{(t)}}\leq 1\right)\leq \PP\left(\overline{Y^{(t)}}_{1}\geq u\right)\leq C_1 h^{\mathbf{Y}^{(t)}}(u)
\leq  24C_1t\psi^{*}(\psi^{-1}(1/t)1/u)\\ &\leq 24C_1c(x_0)t\psi(\psi^{-1}(1/t)1/u) =24C_1c(x_0)\frac{\psi(\psi^{-1}(1/t)1/u)}{\psi(\psi^{-1}(1/t))}\leq \frac{48C_1c(x_0)}{u^{\alpha-\delta}}.
\end{align*}
Hence we have
\beq\label{eqn:after M}
\int_{M}^{\eps \psi^{-1}(1/t)}\PP(T_{u}^{\mathbf{Y}^{(t)}}\leq 1)\ud u \leq \int_{M}^{\eps \psi^{-1}(1/t)}\frac{c(d,x_0)}{u^{\alpha-\delta}}\ud u\leq \int_{M}^{\infty}\frac{c(d,x_0)}{u^{\alpha-\delta}}\ud u=\frac{c(d,x_0)}{\alpha-\delta-1}M^{-\alpha+\delta+1}
\eeq
for all $\eps\leq \eps_1$, $M< \eps \psi^{-1}(1/t)$, and $\psi^{-1}(1/t)\geq x_0$.
By letting $t\rightarrow 0$ and then letting $M\rightarrow\infty$ in \eqref{eqn:up to M} 
and \eqref{eqn:after M}, hen the conclusion of the lemma. 
\qed

\begin{lemma}\label{lemma:limit positive distance}
Suppose that $\mathbf{X}$ is a symmetric L\'evy process with the characteristic
exponent $\psi\in\mathcal{R}_\alpha$ for some $\alpha\in(1,2]$.
There exists $\eps_{2}>0$ such that for all $\eps\leq \eps_{2}$,
$$
\lim_{t\rightarrow 0}
\psi^{-1}(1/t)\int_{0}^{\eps}\PP(\overline{X}_{t}\geq x)\ud x
=\EE[\overline{S^{(\alpha)}}_{1}].
$$
\end{lemma}
\pf
Let $\mathbf{Y}^{(t)}$ be the process defined in the proof of Lemma \ref{lemma:hitting a point}.
Recall from the proof of Lemma \ref{lemma:hitting a point} that $\psi^{(t)}$ is the characteristic exponent of $\mathbf{Y}^{(t)}$ and 
\beq\label{eq:charexplim}
\displaystyle\lim_{t\to 0}\psi^{(t)}(\xi)=|\xi|^{\alpha}.
\eeq
Moreover,
$$
\psi^{-1}(1/t)\int^\eps_0\PP\left(\overline{X}_t\geq x\right)\ud x=\psi^{-1}(1/t)\int^\eps_0\PP\left(\overline{Y^{(t)}}_1\geq \psi^{-1}(1/t) x\right)\ud x = \int^{\eps\psi^{-1}(1/t)}_0\PP\left(\overline{Y^{(t)}_1}\geq  x\right)\ud x .
$$
Using \eqref{eq:Pruitt}, \eqref{eq:charexplim}, and  \cite[Theorem VI.5.5]{Gihman},
we can get
that $\overline{Y^{(t)}}_1\xrightarrow{D} \overline{S^{(\alpha)}}_1$ (since $x\mapsto\sup_{t\in[0,1]}x(t)$ is a continuous functional on the Skorohod space). 
The rest of the proof is identical to the proof of Lemma \ref{lemma:hitting a point}.
\qed

\begin{lemma}\label{lemma:positive distance}
Suppose that $\mathbf{X}$ is a symmetric L\'evy process with the characteristic
exponent $\psi\in\mathcal{R}_\alpha$ for some $\alpha\in(1,2]$.
Let $\Omega=\cup_{i}(a_{i},b_{i})$ with $|\Omega|=\sum_{i}(b_{i}-a_{i})<\infty$. 
Suppose that $a_{i}\notin \partial^{ad}\Omega$ and 
$\eps<\frac12((b_{i}-a_{i})\wedge \eps_{2})$. Then we have
$$
\lim_{t\rightarrow 0}\psi^{-1}(1/t)\int_{a_{i}}^{a_{i}+\eps}
\PP_{x}(\tau_{\Omega}\leq t)\ud x=\EE[\overline{S^{(\alpha)}}_{1}].
$$
Similarly, if $b_{i}\notin\partial^{ad}\Omega$ and 
$\eps<\frac12((b_{i}-a_{i})\wedge \eps_{2})$, then
$$
\lim_{t\rightarrow 0}\psi^{-1}(1/t)\int_{b_{i}-\eps}^{b_{i}}
\PP_{x}(\tau_{\Omega}\leq t)\ud x=\EE[\overline{S^{(\alpha)}}_{1}].
$$
\end{lemma}
\pf
By the symmetry of $\textbf{X}$ it is enough to prove the first limit.
Suppose that $a_{i}\notin\partial^{ad}\Omega$, $\eps<\frac{b_{i}-a_{i}}{2}$ and $x\in (a_i, a_i+\eps)$.
Note that, under $\PP_x$, the event $\{\tau_{\Omega}>t\}$ can be written as
\begin{align*}
\{\tau_{\Omega}>t\}
&=\{a_{i}<\underline{X}_t\le \overline{X}_t<a_{i}+2\eps\}
\cup \{\tau_{\Omega}>t, \underline{X}_t< a_{i}\}
\cup \{\tau_{\Omega}>t, \overline{X}_t\geq a_{i}+2\eps\}.
\end{align*}
Note that the first event  of the display above is disjoint with the union of the last two events.
Hence we have
\begin{align*}
\PP_{x}(\tau_{\Omega}>t)
=\ \PP_{x}(a_{i}<\underline{X}_t\le \overline{X}_t<a_{i}+2\eps)
+\ \PP_{x}(\{\tau_{\Omega}>t, \underline{X}_t< a_{i}\}
\cup \{\tau_{\Omega}>t, \overline{X}_t\geq a_{i}+2\eps\}).
\end{align*}
This implies that 
\begin{align}\label{eqn:positive2}
&\PP_{x}(\tau_{\Omega}\leq t)\\
&=\ 1-\PP_{x}(a_{i}<\underline{X}_t\le \overline{X}_t<a_{i}+2\eps)
-\ \PP_{x}(\{\tau_{\Omega}>t, \underline{X}_t\leq a_{i}\}
\cup \{\tau_{\Omega}>t, \overline{X}_t\geq a_{i}+2\eps\})\nn\\
&= \PP_{x}(\{ \underline{X}_t\leq a_{i}\}\cup \{ \overline{X}_t\geq a_{i}+2\eps\} )
-\ \PP_{x}(\{\tau_{\Omega}>t, \underline{X}_t\leq a_{i}\}
\cup \{\tau_{\Omega}>t,\overline{X}_t\geq a_{i}+2\eps\})\nn\\
&=\ \PP_{x}(\underline{X}_t\leq a_{i})+\ \PP_{x}(\overline{X}_t\geq a_{i}+2\eps) 
-\ \PP_{x}(\overline{X}_t\geq a_{i}+2\eps \text{ and } \underline{X}_t\leq a_{i})\nn\\
&\quad  -\ \PP_{x}(\tau_{\Omega}>t, \underline{X}_t\leq a_{i})
-\ \PP_{x}(\tau_{\Omega}>t, \overline{X}_t\geq a_{i}+2\eps)\nn\\
& \quad 
+\ \PP_{x}(\tau_{\Omega}>t, \underline{X}_t\leq a_{i}\text{ and } \overline{X}_t\geq a_{i}+2\eps).\nn
\end{align}
Let $b:=\sup\{x\in \Omega: x< a_{i}\}$. Since $a_{i}\notin\partial^{ad}\Omega$, we have either $\{x\in \Omega: x< a_{i}\}=\emptyset$ thus $b=-\infty$ or $b<a_{i}$.
Hence we have either
$$
\{\tau_{\Omega}>t, \underline{X}_t\leq a_{i}\}=\emptyset
$$
or
\beq\label{eqn:positive3}
\{\tau_{\Omega}>t, \underline{X}_t\leq a_{i}\}=\{\tau_{\Omega}>t, \underline{X}_t\leq b\}.
\eeq
We will deal with the second case since the first case is similar and much easier. 
It follows from \eqref{eqn:positive3} that \eqref{eqn:positive2} can be written as
\begin{align*}
\PP_{x}(\tau_{\Omega}\leq t)
=&\ \PP_{x}(\underline{X}_t\leq a_{i})+\ \PP_{x}(\overline{X}_t\geq a_{i}+2\eps)
 - \PP_{x}(\overline{X}_t\geq a_{i}+2\eps \text{ and } \underline{X}_t\leq a_{i})\nn\\
& - \PP_{x}(\tau_{\Omega}>t,\underline{X}_t\leq b)
 - \PP_{x}(\tau_{\Omega}>t, \overline{X}_t\geq a_{i}+2\eps)\\
& +\ \PP_{x}(\tau_{\Omega}>t, \underline{X}_t\leq a_{i}\text{ and } \overline{X}_t\geq a_{i}+2\eps ).
\end{align*}
Hence 
\begin{eqnarray*}
\PP_{x}(\underline{X}_t\leq a_{i})-\ 2\PP_{x}(\overline{X}_t\geq a_{i}+2\eps)-\PP_{x}(\underline{X}_t\leq b)&\leq &\PP_{x}(\tau_{\Omega}\leq t)\\&\leq& \PP_{x}(\underline{X}_t\leq a_{i})+\ 2\PP_{x}(\overline{X}_t\geq a_{i}+2\eps).
\end{eqnarray*}
Note that by the symmetry of $\mathbf{X}$ we have
$$
\int_{a_{i}}^{a_{i}+\eps}
\PP_{x}(\underline{X}_t\leq a_{i})\ud x
=\int_{0}^{\eps}\PP(\overline{X}_{t}\geq y)\ud y.
$$
Hence from Lemma \ref{lemma:limit positive distance},
$$
\lim_{t\rightarrow 0}\psi^{-1}(1/t)\int_{a_{i}}^{a_{i}+\eps}
\PP_{x}(\underline{X}_t\leq a_{i})\ud x
=\EE[\overline{S^{(\alpha)}}_{1}].
$$
By \eqref{eq:Pruitt} we have
$$
\int_{a_{i}}^{a_{i}+\eps}\PP_{x}(\overline{X}_t\geq a_{i}+2\eps)\ud x
=\int_{\eps}^{2\eps}\PP(\overline{X}_t\ge y)\ud y
\leq C_1 t \eps h(\eps).
$$
Since $\psi^{-1}\in \mathcal{R}_{1/\alpha}$ (see \cite[Theorems 1.5.3 and 1.5.12]{bgt}), we have
$$
\lim_{t\rightarrow 0}\psi^{-1}(1/t)\int_{a_{i}}^{a_{i}+\eps}
\PP_{x}(\overline{X}_t\geq a_{i}+2\eps)\ud x=0.
$$
Since $a_i-b>0$,  by the symmetry of $\mathbf{X}$ and the same argument as above we get
\begin{align*}
\lim_{t\rightarrow 0}\psi^{-1}(1/t)\int_{a_{i}}^{a_{i}+\eps}\PP_{x}(\underline{X}_t\leq b)\ud x=0.
\end{align*}
The proof is now complete.
\qed

\begin{lemma}\label{lemma:zero distance}
Suppose that $\mathbf{X}$ is a symmetric L\'evy process with the characteristic
exponent $\psi\in\mathcal{R}_\alpha$ for some $\alpha\in(1,2]$.
Let $\Omega=\cup_i(a_i,b_i)$ 
with $|\Omega|=\sum_{i}(b_{i}-a_{i})<\infty$. 
If $a_{i}\in\partial^{ad}\Omega$ and 
$\eps<\frac12((b_{i}-a_{i})\wedge \eps_1)$, then
$$
\lim_{t\rightarrow 0}\psi^{-1}(1/t)\int_{a_{i}}^{a_{i}+\eps}
\PP_{x}(\tau_{\Omega}\leq t)\ud x
=\int_{0}^{\infty}\PP(T_{u}^{(\alpha)}\leq 1)\ud u.
$$
Similarly, if $b_{i}\in\partial^{ad}\Omega$ and 
$\eps<\frac12((b_{i}-a_{i})\wedge \eps_1)$, then
$$
\lim_{t\rightarrow 0}\psi^{-1}(1/t)\int_{b_{i}-\eps}^{b_{i}}
\PP_{x}(\tau_{\Omega}\leq t)\ud x
=\int_{0}^{\infty}\PP(T_{u}^{(\alpha)}\leq 1)\ud u.
$$
\end{lemma}
\pf
Suppose that $a_{i}\in\partial^{ad}\Omega$, $\eps<\frac{b_{i}-a_{i}}{2}$ and $x\in (a_i, a_i+\eps)$.
Let $(a_j, b_j)$ with $b_j=a_i$ be the component of $\Omega$ which is adjacent to $(a_i, b_i)$.
Then we have under $\PP_x$,
\begin{eqnarray*}
\{\tau_{\Omega}\leq t\}
&=&\{T_{a_{i}}\leq t\}\cup \{\tau_{\Omega}\leq t, T_{a_{i}}>t\}\subset\{T_{a_{i}}\leq t\}\cup \{\overline{X}_{t}\geq a_{i}+
2\eps \text{ or } \underline{X}_{t}\leq a_j\}.
\end{eqnarray*}
It follows from an argument similar to that in the proof of  Lemma \ref{lemma:positive distance} we have
$$
\lim_{t\rightarrow 0}\psi^{-1}(1/t)\int_{a_{i}}^{a_{i}+\eps}
\PP_{x}(\overline{X}_{t}\geq a_{i}+
2\eps \text{ or } \underline{X}_{t}\leq a_j)\ud x=0.
$$
Hence we have
$$
\limsup_{t\rightarrow 0}\psi^{-1}(1/t)\int_{a_{i}}^{a_{i}+\eps}
\PP_{x}(\tau_{\Omega}\leq t)\ud x\leq 
\limsup_{t\rightarrow 0}\psi^{-1}(1/t)\int_{a_{i}}^{a_{i}+\eps}
\PP_{x}(T_{a_{i}}\leq t)\ud x
$$
and
$$
\liminf_{t\rightarrow 0}\psi^{-1}(1/t)\int_{a_{i}}^{a_{i}+\eps}
\PP_{x}(\tau_{\Omega}\leq t)\ud x\geq 
\liminf_{t\rightarrow 0}\psi^{-1}(1/t)\int_{a_{i}}^{a_{i}+\eps}
\PP_{x}(T_{a_{i}}\leq t)\ud x.
$$
Now using Lemma \ref{lemma:hitting a point} we obtain the claim.
\qed

Now we state a result handling the case when $\Omega$ has infinitely many components.
\begin{lemma}\label{lemma:many components1}
Suppose that $\mathbf{X}$ is a symmetric L\'evy process with the characteristic
exponent $\psi\in\mathcal{R}_\alpha$ for some $\alpha\in(1,2]$.
If $\Omega\subset \R$ is of finite Lebesgue measure and has infinitely many components, then
$$
\liminf_{t \rightarrow 0}\psi^{-1}(1/t)\left(|\Omega|-Q_{\Omega}(t)\right)=\infty.
$$
\end{lemma}
\pf
If $\Omega$ has infinitely many components, either $A$, the number of components in $\tilde{\Omega}$, or $B$, 
the number of points in $\partial^{ad}\Omega$, is
infinite. 
Suppose that $A=\infty$. Let $\Omega=\cup_{i=1}^{\infty}(a_{i},b_{i})$. 
Then there must be infinitely many $i$ such that $a_{i}\notin\partial^{ad}\Omega$ or $b_{i}\notin\partial^{ad}\Omega$.
Let $\mathcal{I}=\{i: a_{i}\notin\partial^{ad}\Omega \text{ or } b_{i}\notin\partial^{ad}\Omega\}$. 
Given $N$, take $\eps=\eps(N)$ small so that 
there are at least $N$ many $i$'s 
with $\eps<\frac12((b_{i}-a_{i})\wedge \eps_1)$.
Then it follows from Lemma \ref{lemma:positive distance} we have
\begin{align*}
&\liminf_{t\rightarrow 0}\psi^{-1}(1/t)\left(|\Omega|-Q_{\Omega}(t)\right)\\
&\geq\liminf_{t\rightarrow 0}\left( \sum_{i\in\mathcal{I}}\psi^{-1}(1/t)\left(\int_{a_{i}}^{a_{i}+\eps}\PP_{x}(\tau_{\Omega}\leq t)\ud x+\int_{b_{i}-\eps}^{b_{i}}\PP_{x}(\tau_{\Omega}\leq t)\ud x\right)\right)\\
&\geq N\EE[\overline{S^{(\alpha)}}_{1}].
\end{align*}
Now the assertion follows by letting $N\rightarrow\infty$.

The case when $B=\infty$ can be proved in a similar way using Lemma \ref{lemma:zero distance}.
\qed

Now we are ready to prove Theorem \ref{thm:general open sets}.

\textbf{Proof of Theorem \ref{thm:general open sets}}
If $\Omega$ has infinitely many components, the result follows from Lemma \ref{lemma:many components1}.
Now assume that $\Omega$ has finitely many components. 
Write $\Omega=\cup_{i=1}^{N}(a_{i},b_{i})$ and let 
$\eps=\frac12\min_{1\leq i\leq N}\left((b_{i}-a_{i})\wedge \eps_1\wedge\eps_2\right)$,
where $\eps_1$ and $\eps_2$ are constants in Lemmas \ref{lemma:hitting a point} and \ref{lemma:limit positive distance}, respectively.
Let $\mathcal{B}$ be the set of points which are the common end point of two adjacent components of $\Omega$
and $\mathcal{A}=\cup_{i=1}^{N}\{a_{i},b_{i}\}\setminus \mathcal{B}$. 
Then $|\mathcal{A}|=2A$, $|\mathcal{B}|=B$, and $A+B=N$.
It follows from 
Lemmas \ref{lemma:inside}, \ref{lemma:positive distance} and \ref{lemma:zero distance} that
\begin{align*}
&\lim_{t\rightarrow 0}\psi^{-1}(1/t)\left(|\Omega|-Q_{\Omega}(t)\right)\\
&=\lim_{t\rightarrow 0}\psi^{-1}(1/t)\left(\int_{\Omega\setminus \Omega_{\eps}}
\PP_{x}(\tau_{\Omega}\leq t)\ud x\right)\\
&\quad +\lim_{t\rightarrow 0}\psi^{-1}(1/t)\sum_{a_{i},b_{k}\in \mathcal{A}}\left(\int_{a_{i}}^{a_{i}+\eps}\PP_{x}(\tau_{\Omega}\leq t)\ud x+\int_{b_{k}-\eps}^{b_{k}}\PP_{x}(\tau_{\Omega}\leq t)\ud x\right)\\
&\quad + \lim_{t\rightarrow 0}\psi^{-1}(1/t)\sum_{a_{j},b_{j-1}\in \mathcal{B}}\left(\int_{a_{j}}^{a_{j}+\eps}\PP_{x}(\tau_{\Omega}\leq t)\ud x+\int_{b_{j-1}-\eps}^{b_{j-1}}\PP_{x}(\tau_{\Omega}\leq t)\ud x\right)\\
&=2A\EE[\overline{S^{(\alpha)}}_{1}]+2BC_{1}.
\end{align*}
\qed

\subsection{Cauchy process}
First we give some preliminary results to prepare for the proof of 
Theorem \ref{thm:mainCauchy}. In this subsection
we assume that  $\mathbf{X}$ is a 
Cauchy process, that is, a symmetric 1-stable L\'{e}vy process, in $\R$ with the characteristic exponent $\psi(\xi)=|\xi|$. 

\begin{lemma}\label{lemma:positive distance 2}
Suppose that $\mathbf{X}$ is a Cauchy process.
Let $\Omega=\cup_i(a_{i},b_{i})$ 
with $|\Omega|=\sum_{i}(b_{i}-a_{i})<\infty$. 
If $a_{i}\notin\partial^{ad}\Omega$ and $\eps<\frac{b_{i}-a_{i}}{2}$, then
$$
\lim_{t\rightarrow 0}\frac{\int_{a_{i}}^{a_{i}+\eps}
\PP_{x}(\tau_{\Omega}\leq t)\ud x}{t\ln(1/t)}=\frac{1}{\pi}.
$$
Similarly, if $b_{i}\notin\partial^{ad}\Omega$ and $\eps <\frac{b_{i}-a_{i}}{2}$, then
$$
\lim_{t\rightarrow 0}\frac{\int_{b_{i}-\eps}^{b_{i}}
\PP_{x}(\tau_{\Omega}\leq t)\ud x}{t\ln(1/t)}=\frac{1}{\pi}.
$$
\end{lemma}
\pf
The proof is almost identical to the proof of Lemma \ref{lemma:positive distance} using \cite[Proposition 4.3 (i)]{Valverde3} instead of Lemma \ref{lemma:limit positive distance}, so we omit the details.
\qed

Now we address the issue when $\Omega$ has adjacent components. 
Recall the definition of augmented set $\tilde{\Omega}$ in \eqref{eqn:augmented set}.
It is well known that, when $0<\alpha\leq 1$, a single point is polar for the process hence $T_{x}=\inf\{s: X_{s}=x\}$ is almost surely infinite.  
Hence we have the following result.
\begin{lemma}\label{lemma:boundary1}
If $\mathbf{X}$ is a 
Cauchy process, then
$Q_{\Omega}(t)=Q_{\tilde{\Omega}}(t)$.
\end{lemma}
\pf 
By (\cite[Theoreme 8]{Bret}) $\{0\}$ is polar and therefore $\partial\Omega$ is polar as well. Hence $$\PP_{x}(\tau_{\tilde{\Omega}}>t)=\PP_{x}(\tau_{\Omega}>t)$$ almost surely. This implies the claim. 
\qed

\begin{lemma}\label{lemma:many components2}
Suppose that $\mathbf{X}$ is a Cauchy process. 
If $\Omega$ is of finite Lebesgue measure and
$\tilde{\Omega}=\Omega\cup \partial^{ad}\Omega$ has infinitely many components, then
$$
\liminf_{t \rightarrow 0}\frac{|\Omega|-Q_{\Omega}(t)}{t\ln(1/t)}=\infty.
$$
\end{lemma}
\pf
The proof is very similar to the proof of Lemma \ref{lemma:many components1} using Lemma \ref{lemma:positive distance 2}.
\qed

\begin{theorem}\label{thm:mainCauchy}Let $\mathbf{X}$ be a Cauchy process.
Let $\Omega=\cup_{i}(a_{i},b_{i})$ be an open set in $\R$ with $|\Omega|=\sum_{i}(b_{i}-a_{i})<\infty$.
Let $A$ be the number of components of $\tilde{\Omega}$.
Then we have
$$
\lim_{t\rightarrow 0}\frac{|\Omega|-Q_{\Omega}(t)}{t\ln(1/t)}=
\frac{2A}{\pi}. 
$$
\end{theorem}

\pf
The proof is very similar to the proof of Theorem \ref{thm:general open sets} using Lemmas \ref{lemma:inside}, \ref{lemma:positive distance 2} and \ref{lemma:boundary1}, \ref{lemma:many components2} and we omit the details.
\qed

\section{Perturbation Results}\label{section:perturbation}
In this section, we assume that $\mathbf{Y}$ is a L\'{e}vy process in $\R^d$ 
with the L\'{e}vy triplet $(A,\gamma,\nu^Y)$ such that
$$
\lim_{t\to0}\frac{|\Omega|-Q^{\mathbf{Y}}_{\Omega}(t)}{f(t)}=C, 
$$
where $\lim_{t\to0}\frac{t}{f(t)}=0$.
Throughout this section, the superscript $\mathbf{Y}$ always means quantities corresponding to the process $\mathbf{Y}$.

Now we assume that $\mathbf{X}$ is a L\'{e}vy process in $\R^d$ with the L\'{e}vy triplet $(A,\gamma,\nu)$ 
such that the signed measure $\sigma(\ud x):=\nu(\ud x)-\nu^{Y}(\ud x)$ has finite total variation $m$.

Here is the main result of this section.
\begin{theorem}\label{thm:main2}
We have
$$
\lim_{t\rightarrow 0}\frac{|\Omega|-Q_{\Omega}(t)}{f(t)}=\lim_{t\rightarrow 0}\frac{|\Omega|-Q^{\mathbf{Y}}_{\Omega}(t)}{f(t)}.
$$
\end{theorem}

In order to prove Theorem \ref{thm:main2} we need two  lemmas.
\begin{lemma}\label{lemma:upper}
If $\sigma(\ud x)$ is a nonnegative measure, then for any $t>0$ we have
$$
|\Omega|-Q^{\mathbf{Y}}_{\Omega}(t)\leq e^{mt}\left(|\Omega|-Q_{\Omega}(t)\right).
$$
\end{lemma}
\begin{proof}
Write $X_{t}=Y_{t}+V_{t}$, 
where $\mathbf{V}=(V_t)_{t\ge 0}$ is a compound Poisson
process independent of $\mathbf{Y}$. 
Let $T=\inf\{s\ge 0 : V_{s}\neq 0\}$. 
Since $\tau_{\Omega}^{\mathbf{Y}}$ and $T$ are independent, we have
\begin{align*}
|\Omega|-Q_{\Omega}(t)&=\int_{\Omega}\PP_{x}(\tau_{\Omega}<t)\ud x
\ge \int_{\Omega}\PP_{x}(\tau_{\Omega}<t, T>t)\ud x\\
&=\int_{\Omega}\PP_{x}(\tau_{\Omega}^{\mathbf{Y}}<t, T>t)\ud x=\int_{\Omega}\PP_{x}(\tau_{\Omega}^{\mathbf{Y}}<t)\PP_{x}(T>t)\ud x\\
&=\int_{\Omega}\PP_{x}(\tau_{\Omega}^{\mathbf{Y}}<t)e^{-mt}\ud x
=e^{-mt}\left(|\Omega|-Q_{\Omega}^{\mathbf{Y}}(t)\right).
\end{align*}
This establishes the claim of the lemma.
\end{proof}

\begin{lemma}\label{lemma:lower}
If $\sigma(\ud x)$ is a nonnegative measure, then for any $t>0$ we have
$$
|\Omega|-Q_{\Omega}^{\mathbf{Y}}(t)
\geq |\Omega|-e^{mt}Q_{\Omega}(t)=(|\Omega|-Q_{\Omega}(t))-(e^{mt}-1)Q_{\Omega}(t).
$$
\end{lemma}
\begin{proof}
As in the proof of Lemma \ref{lemma:upper}, we write $X_{t}=Y_{t}+V_{t}$, where $\mathbf{V}$ is a compound Poisson process independent of $\mathbf{Y}$. Then by independence of $\mathbf{Y}$ and $\mathbf{V}$, we have
\begin{align*}
e^{-mt}Q^{\mathbf{Y}}_{\Omega}(t)
&=e^{-mt}\int_{\Omega}\PP_{x}\left(\tau_{\Omega}^{\mathbf{Y}}>t\right)\ud x
=\PP_{x}\left(T>t\right)\int_{\Omega}\PP_{x}\left(\tau_{\Omega}^{\mathbf{Y}}>t\right)\ud x\\
&=\int_{\Omega}\PP_{x}\left(\tau_{\Omega}^{\mathbf{Y}}>t, T>t\right)\ud x
\leq\int_{\Omega}\PP_{x}\left(\tau_{\Omega}>t\right)\ud x
=Q_{\Omega}(t),
\end{align*}
where we used the fact that $\{\tau_{\Omega}^{\mathbf{Y}}>t, T>t\}
\subset \{\tau_{\Omega}>t\}$ in the last inequality. 
Hence we have $Q^{\mathbf{Y}}_{\Omega}(t)\leq e^{mt}Q_{\Omega}(t)$ and this immediately implies $|\Omega|-Q^{\mathbf{Y}}_{\Omega}(t)\geq |\Omega|-e^{mt}Q_{\Omega}(t)$.
\end{proof}

Now we are ready to prove Theorem \ref{thm:main2}.

{\bf Proof of Theorem \ref{thm:main2}}.
By assumption the signed measure $\sigma(\ud x)$ has finite total variation.
Let $\sigma(\ud x)=\sigma^{+}(\ud x)-\sigma^{-}(\ud x)$ be the 
Hahn-Jordan decomposition 
(see \cite[Theorem 3.3 and 3.4]{Folland}) of $\sigma(\ud x)$ 
such that $\mathcal{P}\cup \mathcal{N}=\R^{d}$, $\mathcal{P}\cap \mathcal{N}=\emptyset$, and $\sigma^{+}(\mathcal{N})=\sigma^{-}(\mathcal{P})=0$. 
Let $\mathbf{Z}$ be a L\'evy process with L\'evy density $\nu^{\mathbf{Z}}(\ud x)=\nu(\ud x)1_{\mathcal{N}}(x)+\nu^{\mathbf{Y}}(\ud x)1_{\mathcal{P}}(x)$.
Note that $\sigma_{+}(\ud x):=\nu(\ud x)-\nu^{\mathbf{Z}}(\ud x)$ is a nonnegative measure on $\R^{d}$ and 
$$
m_{1}:=\int_{\R^{d}}\sigma_{+}(\ud x)=\int_{\mathcal{P}}\sigma^{+}(\ud x)= \|\sigma^{+}\|\leq \|\sigma\|<\infty.
$$
Hence from Lemmas \ref{lemma:upper} and \ref{lemma:lower} we have
\begin{eqnarray}\label{eqn:pertubation inequality1}
&&|\Omega|-Q_{\Omega}^{\mathbf{Z}}(t)
\leq e^{m_{1}t}\left(|\Omega|-Q_{\Omega}(t)\right),\nn\\
&&|\Omega|-Q_{\Omega}^{\mathbf{Z}}(t)
\geq |\Omega|-Q_{\Omega}(t)-(e^{m_{1}t}-1)|\Omega|.
\end{eqnarray}

By interchanging the role of $\mathbf{X}$ and $\mathbf{Y}$ we also have 
\begin{eqnarray}
&&|\Omega|-Q_{\Omega}^{\mathbf{Z}}(t)\leq e^{m_{2}t}\left(|\Omega|-Q_{\Omega}^{\mathbf{Y}}(t)\right),
\nn\\
&&|\Omega|-Q_{\Omega}^{\mathbf{Z}}(t)\geq |\Omega|-Q_{\Omega}^{\mathbf{Y}}(t)-(e^{m_{2}t}-1)|\Omega|, \label{eqn:pertubation inequality2b}
\end{eqnarray}
where $\sigma_{-}(\ud x):=\nu^{\mathbf{Y}}(\ud x)-\nu^{\mathbf{Z}}(\ud x)$ and $m_{2}:=\int_{\R^{d}}\sigma_{-}(\ud x)=\int_{\mathcal{N}}\sigma^{-}(\ud x)<\infty.$

Hence it follows from \eqref{eqn:pertubation inequality1} and \eqref{eqn:pertubation inequality2b} we have
$$
|\Omega|-Q_{\Omega}^{\mathbf{Y}}(t)\leq |\Omega|-Q_{\Omega}^{\mathbf{Z}}(t)+(e^{m_{2}t}-1)|\Omega|\leq e^{m_{1}t}(|\Omega|-Q_{\Omega}(t))+(e^{m_{2}t}-1)|\Omega|,
$$
and 
$$
|\Omega|-Q_{\Omega}^{\mathbf{Y}}(t)\geq e^{-m_{2}t}(|\Omega|-Q_{\Omega}^{\mathbf{Z}}(t))\geq e^{-m_{2}t}\left(|\Omega|-Q_{\Omega}(t)-(e^{m_{1}t}-1)|\Omega|\right).
$$
Since $\lim_{t\rightarrow 0}\frac{e^{mt}-1}{f(t)}=0$, we have
$$
\lim_{t\rightarrow 0} \frac{|\Omega|-Q_{\Omega}^{\mathbf{Y}}(t)}{f(t)}
\leq \liminf_{t\rightarrow 0} \frac{|\Omega|-Q_{\Omega}(t)}{f(t)},
$$
and 
$$
\lim_{t\rightarrow 0} \frac{|\Omega|-Q_{\Omega}^{\mathbf{Y}}(t)}{f(t)}
\geq \limsup_{t\rightarrow 0} \frac{|\Omega|-Q_{\Omega}(t)}{f(t)}.
$$
The proof is now complete.
\qed

\section{Examples}\label{section:examples}
In this section we examine concrete examples of the asymptotic behavior of the spectral heat content for various L\'evy processes.

\noindent 1. 
\textbf{Symmetric stable processes in $\R$ and their perturbations.}

Recall that the L\'evy measure of the symmetric $\alpha$-stable process in $\R$ is
given by 
$\nu^{\mathbf{S}^{(\alpha)}}(\ud x)=\frac{c(1,\alpha)}{|x|^{1+\alpha}}\ud x$.

Now we assume that $\mathbf{X}$ is a L\'evy process in $\R$ with 
L\'evy triplet  $(0,0,\nu)$ such that the signed measure $\sigma(\ud x) =
\nu(\ud x)-\nu^{\mathbf{S}^{(\alpha)}}(\ud x)$ has finite total variation $m$.
Let $$
f_{\alpha}(t)=\begin{cases}
t^{1/\alpha} &\mbox{if } 1<\alpha\leq 2,\\
t\ln \frac{1}{t} &\mbox{if } \alpha=1,\\
t&\mbox{if } 0<\alpha<1.
\end{cases}
$$

Note that, when $0<\alpha<1$, the process $\textbf{X}$ is of bounded variation.
As a consequence of Corollary \ref{cor:HK 1d}, Theorems \ref{thm:general open sets}, \ref{thm:mainCauchy}, and \ref{thm:main2}, we immediately get the following.

\begin{proposition}\label{prop:st1}
Suppose the assumptions in the paragraph above hold. 
Let $\Omega=\cup_{i}(a_{i},b_{i})$ be an open set in $\R$ with $|\Omega|=\sum_{i}(b_{i}-a_{i})<\infty$.
Let $A$ be the number of components of $\tilde{\Omega}$ and $B$ be number 
of points in $\partial^{ad}\Omega$.
Then we have
$$
\lim_{t\rightarrow 0}
\frac{|\Omega|-Q_{\Omega}(t)}{f_{\alpha}(t)}= 
\begin{cases}
2A\EE[\overline{{S}^{(\alpha)}}_{1}]
+2BC_{1}, &\mbox{if } 1<\alpha\leq 2,\\
\frac{2A}{\pi}, \quad \alpha=1,\\
\rm{Per}_{\mathbf{X}}(\Omega), \quad 0<\alpha<1.
\end{cases}
$$
\end{proposition}

\begin{remark}
Proposition \ref{prop:st1} is a natural generalization of the main result in \cite{Valverde3}. We remark here that the set $\Omega\subset\R$ is an arbitrary open in $\R$ of finite Lebesgue measure. 
The class of process we are dealing with here is much larger 
than the class of symmetric stable processes.
\end{remark}

\noindent 2. 
\textbf{Fractional perimeter for symmetric stable processes in $\R$.}

\begin{enumerate}
\item [(i)]
If $\Omega\subset \R$ has finitely many components, then $\Omega$ has a finite perimeter, which is equivalent to $f_{\Omega}(y)$ is Lipschitz (see \cite{Galerne}), we have
$\int_{\R^{d}}\frac{f_{\Omega}(y)}{|y|^{d+\alpha}}\ud y<\infty$. 
Hence Theorem \ref{thm:HK on rough sets} recovers \cite[Theorem 3]{cg}.

\item[(ii)]\label{example:diverge} 
Now we give an example of an open set 
$\Omega\subset \R$ with $|\Omega|<\infty$ such that 
$\Per_{\mathbf{S}^{(\alpha)}}(\Omega)=\infty$
for all $0<\alpha<1$. 
Let $\{d_{n}\}_{n\in\mathbb{N}}\subset(0,1)$ be a strictly decreasing sequence such that 
$\sum_{n=1}^{\infty}d_{n}<\infty$. We consider an open set $
\Omega:=\cup_{n=1}^{\infty}(n,n+d_{n}).
$
Define for each $y\in (0, 1)$ a number
$
n(y)=\sup\{k:d_{k}\geq y\}.
$
We have
$
g_{\Omega}(y)
=(d_{1}+d_{2}+\cdots +d_{n(y)})-n(y)y,
$
and
$$
f_\Omega(y)=\sum_{k=n(y)+1}^{\infty}d_{k}+n(y)y\geq n(y)y.
$$
Now we fix $b>1$ and consider $d_{n}=\frac{1}{n(1+\ln n)^{b}}$.
Using the definition of $n(y)$ it is easy to check that
\begin{equation}\label{eqn:comparison1}
f_{\Omega}(y)\geq n(y)y \geq    \frac{1}{2(2+ \ln n(y))^b}\geq c(b)\ln^{-b}(y^{-1}), \quad y\in(0,1/4).
\end{equation}
Thus for any $0<\alpha<1$ we have
$$
\Per_{\mathbf{S}^{(\alpha)}}(\Omega)\geq\int_{\{|y|\leq 1/4\}}\frac{f_{\Omega}(y)}{|y|^{1+\alpha}}\ud y=\infty.
$$

\item[(iii)] 
Finally we state a simple criteria that guarantees $\Per_{\mathbf{S}^{(\alpha)}}(\Omega)<\infty$. 
\begin{lemma}Let $\alpha\in(0,1)$.
Suppose that $\Omega=\cup_{i=1}^{\infty}\Omega_{i}\subset \R$, where $\Omega_i$ are open connected and disjoint. 
If $\sum_{i=1}^{\infty}|\Omega_{i}|^{1-\alpha}<\infty$, then
$$
\Per_{\mathbf{S}^{(\alpha)}}(\Omega)<\infty. 
$$
\end{lemma}
\pf
Note that for $x\in \Omega$ we have
\begin{eqnarray*}
&&\int_{\Omega^{c}} \frac{1}{|x-y|^{1+\alpha}}\ud y\leq 2\int_{\delta_{\Omega}(x)}^{\infty}\frac{1}{r^{1+\alpha}}=\frac{2}{\alpha}\delta_{\Omega}(x)^{-\alpha}.
\end{eqnarray*}
Hence
\begin{align*}
&\Per_{\mathbf{S}^{(\alpha)}}\Omega=\int_{\Omega}\int_{\Omega^{c}}\frac{1}{|x-y|^{1+\alpha}}\ud y\ud x\leq\int_{\Omega}\frac{2}{\alpha}\delta_{\Omega}(x)^{-\alpha}\ud x
\label{eqn:finite1}\\
&=\sum_{i=1}^{\infty}\int_{\Omega_{i}}\frac{2}{\alpha}\delta_{\Omega}(x)^{-\alpha}\ud x=\sum_{i=1}^{\infty}\frac{4}{\alpha}\int_{0}^{\frac{|\Omega_{i}|}{2}}\frac{\ud r}{r^{\alpha}}
=\sum_{i=1}^{\infty}\frac{4}{\alpha(1-\alpha)}\left(\frac{|\Omega_{i}|}{2}\right)^{1-\alpha}.\nn
\end{align*}
\qed

Let consider $
\Omega=\cup_{n=1}^{\infty}(n,n+\frac{1}{n^{b}})
$
with $b>1$. 
By the above lemma $\Per_{\mathbf{S}^{(\alpha)}}\Omega<\infty$ if $b>1/(1-\alpha)$.
Hence if $b>1/(1-\alpha)$, then we have
$$
\lim_{t\rightarrow 0}\frac{|\Omega|-Q_{\Omega}(t)}{t}<\infty .
$$
On the other hand, using an argument similar 
to that leading to \eqref{eqn:comparison1}, we get
  $\Per_{\mathbf{S}^{(\alpha)}}\Omega=\infty$, for $b\leq 1/(1-\alpha)$.
Hence we conclude that 
$$
\lim_{t\rightarrow 0}\frac{|\Omega|-Q_{\Omega}(t)}{f_{\alpha}(t)}<\infty \text{ if and only if } b>1/(1-\alpha).
$$
\end{enumerate}

\noindent 3. 
\textbf{Isotropic $\alpha$-stable processes in $\R^{d}$, $\alpha\in (0, 1)$ and $d\geq 2$.}

Consider an isotropic $\alpha$- stable process
$\mathbf{S}^{(\alpha)}$ with $0<\alpha<1$ on $\R^{d}$, $d\geq 2$.
Suppose that the open set $\Omega$ satisfies the following volume density condition (see \cite[Equation (1.1)]{Wu})
$$
\left|\Omega \cap B(x, 2d(x,\partial \Omega))\right|>c\, \mathrm{dist}(x,\partial D)^{d}
$$
for some constant $c>0$ and $|\partial \Omega|=0$. Then it follows from \cite[Theorem 1]{Wu} that $\PP_{x}(\tau_{\Omega}^{\mathbf{S}^{(\alpha)}}\in \partial \Omega)=0$ for all $x\in \Omega$. 
Hence by Theorem \ref{thm:main7} we have
$$
\lim_{t\to 0}\frac{|\Omega|-Q^{\mathbf{S}^{(\alpha)}}_{\Omega}(t)}{t}=\rm{Per}_{\mathbf{S}^{(\alpha)}}(\Omega).
$$
Note that any Lipschitz open sets satisfy volume density condition.

\noindent 4. \textbf{Relativistic stable processes.}

Suppose that $\mathbf{X}^{m}$ is a relativistic $\alpha$-stable process
with mass $m>0$ whose the characteristic exponent is 
$$
\psi^{m}(\xi)=(|\xi|^{2}+m^{2/\alpha})^{\alpha/2}-m, \quad \xi\in \R^{d}.
$$
Let $\nu^{m}(x)$ be the L\'evy density of $\mathbf{X}^{m}$. 
It is well-known that
$0<\nu^{m}(x)\leq \nu^{\mathbf{S}^{(\alpha)}}(x)$ and 
$$\int_{\R^{d}}\left(\nu^{\mathbf{S}^{(\alpha)}}(x)-\nu^{m}(x)\right)\ud x=m<\infty.
$$
Hence if $\alpha \geq 1$ and $d=1$, it follows from Proposition \ref{prop:st1} that
$$
\lim_{t\rightarrow 0}
\frac{|\Omega|-Q^{\mathbf{X}^{m}}_{\Omega}(t)}{f_{\alpha}(t)}=
\begin{cases}
2A\EE[\overline{S^{(\alpha)}}_{1}]
+2BC_{1}, &\mbox{if } 1<\alpha<2,\\
\frac{2A}{\pi}, \quad \alpha=1.
\end{cases}
$$

On the other hand when $0<\alpha<1$ and $\Omega$ is a Lipschitz open set in $\R^{d}$, $d\geq 2$ or an arbitrary open set in $\R$ it follows from Corollaries \ref{cor:IsoSHC}  and \ref{cor:HK 1d} that
$$
\lim_{t\to 0}\frac{|\Omega|-Q_{\Omega}^{\mathbf{X^{m}}}(t)}{t}=
\rm{Per}_{\mathbf{X}^{m}}(\Omega).
$$

\noindent 5. \textbf{Truncated stable processes.}

Let $\mathbf{X}^{T}$ be a truncated $\alpha$-stable process
with L\'evy triplet $(0,0,\nu^{\mathbf{X}^{T}})$, where
$$
\nu^{\mathbf{X}^{T}}(x)=\nu^{\mathbf{S}^{(\alpha)}}(x)
\cdot 1_{\{\|x\|\leq 1\}}(x).
$$
By the same argument as in 
the case of relativistic stable processes,  we get that when $\alpha \geq 1$ and $d=1$, we have
$$
\lim_{t\rightarrow 0}
\frac{|\Omega|-
Q^{\mathbf{X}^{T}}_{\Omega}(t)}{f_{\alpha}(t)}= 
\begin{cases}
2A\EE[\overline{S}^{(\alpha)}_{1}]
+2BC_{1}, &\mbox{if } 1<\alpha<2,\\
\frac{2A}{\pi}, \quad \alpha=1.
\end{cases}
$$
When $0<\alpha<1$ and $\Omega$ is a Lipschitz open set 
in  $\R^{d}$, $d\geq 2$ or an arbitrary open set in $\R$ it follows from Corollaries \ref{cor:IsoSHC}  and \ref{cor:HK 1d} that
$$
\lim_{t\to 0}\frac{|\Omega|-Q_{\Omega}^{\mathbf{X}^{T}}(t)}{t}=\rm{Per}_{\mathbf{X}^{T}}(\Omega).
$$

\noindent 6. \textbf{Logarithmic perturbations.}

Let $\mathbf{X}^{L}$ be a L\'evy process  with 
L\'evy triplet $(0,0,\nu^{\mathbf{X}^{L}})$, where
$$
\nu^{\mathbf{X}^{L}}(\ud x)
=(\ln(2+\frac{1}{\|y\|}))^\beta\nu^{\mathbf{S^{(\alpha)}}}(\ud x),
$$
where $\beta\in\R$.
By \cite[Proposition 2]{CGT} we have that $\psi^L\in\mathcal{R}_{\alpha}$ and $\psi^L(s)\sim s^{\alpha}\ln^{\beta}(s)$, where $f(s)\sim g(s)$ means 
$\displaystyle\lim_{s\rightarrow\infty}\frac{f(s)}{g(s)}=1$. 
This and \cite[Proposition 1.5.15]{bgt} imply $(\psi^L)^{-1}(s)\sim s^{1/\alpha}\ln^{-\beta/\alpha}(s)$. 
Hence we get by Theorem \ref{thm:general open sets},  for  $\alpha>1$
$$
\lim_{t\rightarrow 0}
\frac{|\Omega|-Q^{\mathbf{X}^{L}}_{\Omega}(t)}{t^{1/\alpha}\ln^{\beta/\alpha}(1/t)}= 
2A\EE[\overline{S^{(\alpha)}}_{1}]
+2BC_{1}.
$$
When $\alpha<1$ or $\alpha=1$ and $\beta<-1$ the process $\textbf{X}^L$ is of bounded variation, therefore one can apply Theorem \ref{thm:main7} and Corollaries \ref{cor:HK 1d} and \ref{cor:IsoSHC} in this case.

\bibliographystyle{plain}

\begin{thebibliography}{10}

\bibitem{Valverde2}
L.~Acu{\~n}a~Valverde.
\newblock Heat content estimates over sets of finite perimeter.
\newblock {\em J. Math. Anal. Appl.}, \textbf{441} (2016), 104--120.

\bibitem{Valverde1}
L.~Acu{\~n}a~Valverde.
\newblock Heat content for stable processes in domains of $\mathbb{R}^d$.
\newblock {\em J. Geom. Anal.}, \textbf{27}(1) (2017), 492--524.

\bibitem{Valverde3}
L.~Acu{\~n}a~Valverde.
\newblock On the one dimensional spectral heat content for stable processes.
\newblock {\em J. Math. Anal. Appl.}, \textbf{441} (2016), 11--24.

\bibitem{Ambrosio_2000}
L.~Ambrosio, N.~Fusco, and D.~Pallara.
\newblock {\em Functions of bounded variation and free discontinuity problems}.
\newblock Oxford Mathematical Monographs. The Clarendon Press, Oxford
  University Press, New York, 2000.



\bibitem{Ber} J. Bertoin. 
\newblock {\em L\'evy processes}.
\newblock Cambridge University Press, Cambridge 1996.


\bibitem{bgt}
N.H. Bingham, C.M. Goldie, and J.L. Teugels.
\newblock {\em Regulary variation}, volume~27 of {\em Encyclopedia of
  Mathematics and its Applications}.
\newblock Cambridge University Press, 1987.


\bibitem{Bret}
 J. Bretagnolle.
   \newblock R\'esultats de {K}esten sur les processus \`a accroissements
              ind\'ependants. 
 Lecture Notes in Math., Vol. 191: 21--36.


\bibitem{cg}
W.~Cygan and T.~Grzywny.
\newblock Heat content for convolution semigroups.
\newblock  {\em J. Math. Anal. Appl.}, \textbf{446} (2017), 1393--1414.

\bibitem{wctg}
W.~Cygan and T.~Grzywny.
\newblock A note on the generalized heat content for L\'{e}vy processes.
\newblock  preprint 2017, arxiv: 1703.10790.



\bibitem{CGT}
W.~Cygan, T.~Grzywny, and B.~Trojan.
\newblock Asymptotic behavior of densities of unimodal convolution semigroups.
\newblock {\em Trans. Amer. Math. Soc.}, \textbf{369}(8), (2017), 5623--5644. 


\bibitem{Folland}
G. B. ~Folland.
\newblock {\em Real analysis. Modern techniques and their applications},  Second edition. Pure and Applied Mathematics (New York). A Wiley-Interscience Publication. 
\newblock John Wiley \& Sons, Inc., New York, 1999.

\bibitem{MRT} 
J. M. Mazon, J. D. Rossi, and J. Toledo.
The heat content for the nonlocal diffusion with non-singular kernels.
{\em Advanced Nonlinear Studies}, \textbf{17} (2017).
DOI: https://doi.org/10.1515/ans-2017-0005


\bibitem{Fusco}
N.~Fusco, V.~Millot, and M.~Morini.
\newblock A quantitative isoperimetric inequality for fractional perimeters.
\newblock {\em J. Funct. Anal.}, \textbf{261} (2011), 697--715.

\bibitem{Galerne}
B.~Galerne.
\newblock Computation of the perimeter of measurable sets via their
  covariogram. applications to random sets.
\newblock {\em Image Analysis \& Stereology}, \textbf{30} (2011), 39--51.


\bibitem{Gihman} I.I. Gihman  and A.V. Skorohod. {\em The theory of stochastic processes. {I}.}      Springer-Verlag, New York-Heidelberg, 1974.

\bibitem{Grzywny1}
T.~Grzywny.
\newblock On {H}arnack inequality and {H}\"older regularity for isotropic
  unimodal {L}\'evy processes.
\newblock {\em Potential Anal.}, \textbf{41} (2014), 1--29.

\bibitem{MR0142153}
N.~Ikeda and S.~Watanabe.
\newblock On some relations between the harmonic measure and the {L}\'evy
  measure for a certain class of {M}arkov processes.
\newblock {\em J. Math. Kyoto Univ.}, \textbf{2} (1962), 79--95.


\bibitem{Miranda1}
M.~Jr. Miranda, D.~Pallara, F.~Paronetto, and M.~Preunkert.
\newblock On a characterisation of perimeters in $\mathbb{R}^n$ via heat
  semigroup.
\newblock {\em Ric. Mat.}, \textbf{44} (2005), 615–--621.

\bibitem{Miranda2}
M.~Jr. Miranda, D.~Pallara, F.~Paronetto, and M.~Preunkert.
\newblock Short-time heat flow and functions of bounded variation in
  $\mathbb{R}^n$.
\newblock {\em Ann. Fac. Sci. Toulouse}, \textbf{16} (2007), 125--145.

\bibitem{Preunkert}
M.~Preunkert.
\newblock A semigroup version of the isoperimetric inequality.
\newblock {\em Semigroup Forum}, \textbf{68} (2004), 233--245.

\bibitem{Pruitt}
W.E. Pruitt.
\newblock The growth of random walks and {L}\'{e}vy processes.
\newblock {\em Ann. Probab.}, \textbf{9} (1981), 948--956.

\bibitem{Sato}
K.~Sato.
\newblock {\em L\'evy processes and infinitely divisible distributions},
  volume~68 of {\em Cambridge Studies in Advanced Mathematics}.
\newblock Cambridge University Press, 1999.


\bibitem{Sztonyk}
P.~Sztonyk.
On harmonic measure for L\'evy processes. {\em Probab. Math. Stat.}
\textbf{20} (2000), 383--390.


\bibitem{vanDenBerg_4}
M.~van~den Berg.
\newblock Heat content and {B}rownian motion for some regions with a fractal
  boundary.
\newblock {\em Probab. Theory Rel. Fields}, \textbf{100} (1994), 439--456.

\bibitem{vanDenBerg1_POT}
M.~van~den Berg.
\newblock Heat flow and perimeter in {$\mathbb{R}^m$}.
\newblock {\em Potential Anal.}, \textbf{39} (2013), 369--387.

\bibitem{vanDenBerg3}
M.~van~den Berg, E.~B. Dryden, and T.~Kappeler.
\newblock Isospectrality and heat content.
\newblock {\em Bull. Lond. Math. Soc.}, \textbf{46} (2014), 793--808.

\bibitem{vanDenBerg1}
M.~van~den Berg and P.~Gilkey.
\newblock Heat flow out of a compact manifold.
\newblock {\em J. Geom. Anal.}, \textbf{25} (2015), 1576--1601.

\bibitem{vanDenBerg2}
M.~van~den Berg and K.~Gittins.
\newblock Uniform bounds for the heat content of open sets in {E}uclidean
  space.
\newblock {\em Differential Geom. Appl.}, \textbf{40} (2015), 67--85.

\bibitem{vanDenBerg5}
M.~van~den Berg and J.F. Le~Gall.
\newblock Mean curvature and the heat equation.
\newblock {\em Math. Z.}, \textbf{215} (1994), 437--464.


\bibitem{Wu}
J. M. Wu.
\newblock Harmonic measures for symmetric stable processes.
\newblock {\em Studia. Math.}. \textbf{149} no 3. (2002), 281--293. 


\end{thebibliography}

\end{document}